\documentclass[leqno,11pt]{amsart}

\usepackage{geometry}

\usepackage[all]{xy} 

\usepackage{mathtools}
\usepackage{amsmath, amssymb, amsfonts, latexsym, mdwlist, amsthm}
\usepackage{subfig}
\usepackage{graphicx}
\usepackage{wrapfig}

\usepackage[bookmarks, colorlinks, breaklinks, pdftitle={},
pdfauthor={}]{hyperref}
\hypersetup{linkcolor=blue,citecolor=blue,filecolor=black,urlcolor=blue}

\usepackage{tikz}
\usetikzlibrary{calc,trees,positioning,arrows,chains,shapes.geometric,%
    decorations.pathreplacing,decorations.pathmorphing,shapes,%
    matrix,shapes.symbols}

\tikzset{
>=stealth',
  punktchain/.style={
    rectangle,
    rounded corners,
    draw=black, thick,
    minimum height=3em,
    text centered,
    on chain},
  line/.style={draw, thick, <-},
  element/.style={
    tape,
    top color=white,
    bottom color=blue!50!black!60!,
    minimum width=8em,
    draw=blue!40!black!90, very thick,
    text width=10em,
    minimum height=3.5em,
    text centered,
    on chain},
  every join/.style={->, thick,shorten >=1pt},
  decoration={brace},
  tuborg/.style={decorate},
  tubnode/.style={midway, right=2pt},
}

\usepackage{paralist}
\setdefaultenum{(a)}{(i)}{}{}
\usepackage{enumitem} 
\usepackage{graphicx}

\usetikzlibrary{patterns}

\renewcommand\_{^{}_}
\newcommand\To{\longrightarrow}
\newcommand{\Into}{\ensuremath{\lhook\joinrel\relbar\joinrel\rightarrow}}
\newcommand\PP{\mathbb P}

\newcommand\Q{\mathbb Q}
\newcommand\R{\mathbb R}

\newcommand\Z{\mathbb Z}
\newcommand\cA{\mathcal A}
\newcommand\cD{\mathcal D}
\newcommand\cH{\mathcal H}
\newcommand\cO{\mathcal O}
\newcommand\cI{\mathcal I}
\newcommand\cL{\mathcal L}
\newcommand\ch{\operatorname{ch}}
\newcommand\Hom{\operatorname{Hom}}
\newcommand\Ext{\operatorname{Ext}}
\newcommand\rk{\operatorname{rank}}
\newcommand\so{\operatorname{\Longrightarrow}}
\renewcommand\={\ =\ }
\newcommand\arXiv[1]{\href{http://arxiv.org/abs/#1}{arXiv:#1}}
\newcommand\mathAG[1]{\href{http://arxiv.org/abs/math/#1}{math.AG/#1}}

\newcommand\ii{\operatorname{(ii)}}

\makeatletter
\newtheorem*{rep@theorem}{\rep@title}
\newcommand{\newreptheorem}[2]{%
\newenvironment{rep#1}[1]{%
 \def\rep@title{#2 \ref{##1}}%
 \begin{rep@theorem}}%
 {\end{rep@theorem}}}
\makeatother

\newtheorem{Thm}{Theorem}[section]
\newreptheorem{Thm}{Theorem}
\newtheorem{Prop}[Thm]{Proposition}

\newtheorem{Lem}[Thm]{Lemma}

\newtheorem{Cor}[Thm]{Corollary}
\newreptheorem{Cor}{Corollary}
\newtheorem{Con}[Thm]{Conjecture}
\newreptheorem{Con}{Conjecture}

\newtheorem{thm-int}{Theorem}

\theoremstyle{definition}
\newtheorem{Def-s}[Thm]{Definition}
\newtheorem{Def}[Thm]{Definition}
\newtheorem{Rem}[Thm]{Remark}

\newcommand{\ignore}[1]{}

\begin{document}

	\title[An application of wall-crossing to {N}oether--{L}efschetz loci]
	{An application of wall-crossing to {N}oether--{L}efschetz loci\vspace{-2mm}}
\author[S. Feyzbakhsh and R. P. Thomas]{S. Feyzbakhsh and R. P. Thomas, with an appendix by C. Voisin\vspace{-1mm}}

	\begin{abstract}
	Consider a smooth projective 3-fold $X$ satisfying the  Bogomolov--Gieseker conjecture of Bayer-Macr\`{i}-Toda (such as $\PP^3$, the quintic threefold or an abelian threefold).
	
	Let $L$ be a line bundle supported on a very positive surface in $X$. If $c_1(L)$ is a primitive cohomology class then we show it has very negative square.
	\end{abstract}

\vspace*{-12mm}
	\maketitle	
\vspace{-5mm}

	\section{Introduction}
	Let $(X,\mathcal O(1))$ be a smooth polarised complex threefold.  For the strongest results
	we take $\cO(1)$ to be primitive. Set $H \coloneqq c_1(\mathcal O(1))$, though we do not require it to be effective.

Weak stability conditions on the derived category $\cD(X)$ were introduced by Bayer-Macr\`{i}-Toda \cite{BMT}. Together with their Bogomolov-Gieseker Conjecture \ref{conjecture} below they constitute the main technique for producing Bridgeland stability conditions on threefolds.

We only need certain weakenings of the conjecture described in \ref{BG1},\,\ref{BG2} below. They are known to hold for many threefolds  \cite{BMS, KosekiAb, Kosekinef, Li, Li.Fano, MP, Ma, ScQ} such as $\PP^3$ or the quintic 3-fold. We apply them to certain weak-semistable objects of $\cD(X)$ as we move
through the space of weak stability conditions. Combined with wall-crossing techniques this proves results about line bundles on surfaces in $|\cO(n)|$.
	

    \begin{Thm}\label{main}
Fix any irreducible divisor\footnote{$D$ may be singular. The results also apply to $D$ reducible, so long as $L$ is \emph{slope semistable} on $D$.} $D \subset X$ in $|\mathcal O(n)|$ and any line bundle $L$ on $D$ with $c_1(L)\ne0$ in $H^2(D,\Q)$ and $c_1(L).H=0$.
\begin{enumerate}[label=\textbf{\emph{(\Alph*)}}]
\item\label{A} If \emph{\ref{BG1}} holds on $X$ and $n \geq 4$ 
then $\displaystyle{L^2\ \leq\ \frac{-2n}{3}\,}.$\smallskip
\item\label{B} If \emph{\ref{BG2}} holds on $X$ and $n \geq 10$ then $L^2\ \leq\ -2n +4$.
\end{enumerate}
    \end{Thm}\medskip

See below for consequences of \emph{\ref{B}} on $\PP^3$, for the observation that it is sharp, and for stronger inequalities for line bundles $L=\cL|_D$ which are restricted from $X$.

It is the classes on $D$ which are \emph{not} restricted from $X$ that most interest us. One obvious source of such classes is the \emph{vanishing cycles} of $D$ --- the (co)homology classes of the Lagrangian two-spheres in $D$ that are contracted to nodes as we deform $D$ inside $|\cO(n)|$ to a nodal surface. These classes all have square $-2>\frac{-2n}3$ so Theorem \ref{main} tells us they can never be the class of a line bundle $L$ on $D$.

\begin{Cor} 
The vanishing cycles of $D\in|\cO(n)|$ have empty Noether-Lefschetz loci. \\
In fact any sum of $m$ disjoint vanishing cycles has empty Noether-Lefschetz locus when
\begin{itemize}
\item $X$ satisfies \emph{\ref{BG1}}, $n\ge4$ and $m\le\lfloor \frac{n-1}{3} \rfloor$, or
\item $X$ satisfies \emph{\ref{BG2}}, $n\ge10$ and $m\le n-3$.
\end{itemize}
\end{Cor}

In other words, if we look for irreducible $D\in|\cO(n)|$ where our vanishing class has Hodge type $(1,1)$ we should find only singular $D$ on which our cohomology class has ceased to exist (or, considered as a homology class, some part of it has vanished).

So not all classes in $H^2(D,\Z)$ become $(1,1)$ under some deformation inside $|\cO(n)|$, even though those which do generate $H^2(D,\Z)$ over $\Z$ by \cite[p19]{Voi}.

\subsection*{Method} To prove Theorem \ref{main} we move in a space of weak stability conditions on $\cD(X)$, and show that if $L^2>-2n/3$ then the Bogomolov-Gieseker inequality \ref{BG1} implies $\iota_*L$ is unstable in certain regions, where $\iota\colon D\hookrightarrow X$ is the inclusion. We find the wall on which it becomes unstable, where we show it is destabilised by a map from $\iota_*L$ to $T(-n)[1]$, for some line bundle $T$ with torsion $c_1(T)$. Thus by relative Serre duality for the map $\iota$, 
\begin{equation}\label{ending}
\Hom\_X(\iota_*L , T(-n)[1])\ =\ \Hom\_D(L, T|_D)\ \neq\ 0,
\end{equation}
which means $L^* \otimes T|_D$ is effective.\footnote{This also shows that if $c_1(L)$ is torsion then it lifts to $X$. For $D$ smooth this follows already from the Lefschetz hyperplane theorem: $X$ is made from $D$ by attaching $(n\ge3)$-cells, so $H^3(X,D)$ is torsion free.} Since $L.H=0$ this implies $L=T|_D$, so, in particular $c_1(L)=0$ in $H^2(D,\Q)$.

\subsection*{Projective space}
There are two different ways to saturate the inequality \emph{\ref{B}} on $\PP^3$ and hence deduce it is sharp.

Firstly, we can take $D$ to contain disjoint lines $L_1,\,L_2\subset\PP^3$. Their normal bundles inside $D$ are $\cO_{\PP^1}(-n+2)$, so $L:=\cO_D(L_1-L_2)$ satisfies $L.H=0$ and $L^2=-2n+4$.

Secondly, if an irreducible $D\in|\cO_{\PP^3}(n)|,\,n\ge10$, contains
disjoint degree $d\ne1$ \emph{plane curves} $C_1,\,C_2$, then \emph{\ref{B}} applied to $\cO_D(C_1-C_2)$ proves $n\ge d+2$. Thus \emph{\ref{B}} is saturated if $n=d+2$, and it is indeed easy to construct $D\supset C_1,\,C_2$ of any degree $n\ge d+2$.

More generally if $D\in|\cO_{\PP^3}(n)|$ contains disjoint degree $d$ curves $C_1,\,C_2$ of genus $g_1,\,g_2$ then \emph{\ref{B}} applied to $\cO_D(C_1-C_2)$ gives $g_1+g_2\le(n-4)(d-1)$ for $n\ge10$.

\subsection*{Line bundles restricted from $X$}
When $L=\cL|_D$ extends to a line bundle $\cL$ on $X$  with $\cL.H^2=0$ then \emph{\ref{A}} is trivial on \emph{any} $X$. In fact $L^2=n\cL^2.H$ is divisible by $n$ and $<0$ by the Hodge index theorem, so
\begin{equation}\label{n}
L^2\ \le\ -n.
\end{equation}
But then if \ref{BG2} holds, \emph{\ref{B}} gives $\cL^2.nH\le-2n+4$, i.e. any line bundle $\cL$ on $X$ satisfies
$$
\cL.H^2\=0 \quad\so\quad \cL^2.H\ \le\ -2.
$$
This appears to be nontrivial, but not very (the Hodge index theorem already gives $\le-1$). Plugging it back into the argument that gave \eqref{n} strengthens it to
\begin{equation}\label{2n}
L^2\ \le\ -2n.
\end{equation}

\subsection*{Acknowledgements} It is an honour to dedicate this paper to the memory of Sir Michael Atiyah. His influence on our lives and mathematics is enormous and ongoing. It is pleasurable exercise to trace back all of the maths in this paper to his work.

We are most grateful to Claire Voisin. On seeing our paper she immediately saw how to produce closely related results on $\PP^3$ by more classical methods, drafting Appendix \ref{app}. We also thank Arend Bayer, Paolo Cascini, Chunyi Li, Davesh Maulik, Ivan Smith, Yukinobu Toda and a thorough referee for useful comments. Seven years ago Toda \cite{TodaBG} pioneered the techniques used in this paper to prove the famous OSV conjecture from physics; see the companion paper \cite{FT} for a comparison of our methods and results. The second author acknowledges support from EPSRC grant EP/R013349/1.
\setcounter{tocdepth}{1}
\tableofcontents \vspace{-1cm}
	
\section{Weak stability conditions}
In this section, we review the notion of a weak stability condition on the derived category of coherent sheaves on a smooth threefold. The main references are \cite{BMT,BMS}. 

Let $(X,\mathcal O(1))$ be a smooth polarised complex threefold, and $H = c_1(\mathcal O(1))$. Denote the bounded derived category of coherent sheaves on $X$ by $\cD(X)$ and its Grothendieck group by $K(X):=K(\cD(X))$. We define the $\mu\_H$-slope of a coherent sheaf $E$ on $X$ to be
$$
\mu\_H(E)\ :=\ \left\{\!\!\begin{array}{cc} \frac{\ch_1(E).H^2}{\ch_0(E)H^3} & \text{if }\ch_0(E)\ne0, \\
+\infty & \text{if }\ch_0(E)=0. \end{array}\right.
$$
Associated to this slope every sheaf $E$ has a Harder-Narasimhan filtration. Its graded pieces have slopes whose maximum we denote by $\mu_H^+(E)$ and minimum by $\mu_H^-(E)$.

For any $b \in \mathbb{R}$, let $\cA(b)\subset\cD(X)$ denote the abelian category of complexes
	\begin{equation}\label{Abdef}
	\mathcal{A}(b)\ =\ \big\{E^{-1} \xrightarrow{\,d\,} E^0 \ \colon\ \mu_H^{+}(\ker d) \leq b \,,\  \mu_H^{-}(\text{coker}\; d) > b \big\}. 
	\end{equation}
Then $\cA(b)$ is the heart of a t-structure on $\cD(X)$ by \cite[Lemma 6.1]{Br}. Let $w\in\R\setminus\{0\}$. On $\cA(b)$ we have the slope function\footnote{This is called $\nu\_{b,w}$ in \cite[Equation 7]{BMT}, but we reserve $\nu\_{b,w}$ for its rescaling \eqref{scale}.}
\begin{equation*}
N_{b,w}(E)\ :=\ \left\{\!\!\begin{array}{cc} \frac{w\ch_2^{bH}(E).H - \frac{1}{6}w^3\ch_0(E)H^3}{w^2\ch_1^{bH}(E).H^2} & \text{if }\ch_1^{bH}(E).H^2\ne0, \\
+\infty & \text{if }\ch_1^{bH}(E).H^2=0, \end{array}\right.
\end{equation*}
where $\ch^{bH}(E):=\ch(E)e^{-bH}$. When $w>0$ this defines a Harder-Narasimhan filtration on $\cA(b)$ by \cite[Lemma 3.2.4]{BMT}. It will be convenient to replace this with
\begin{equation}\label{scale}
\nu\_{b,w}\ :=\ \sigma N_{b,\sigma}+b, \quad\text{where }\sigma:=\sqrt{6(w-b^2/2)},
\end{equation}
for $w>b^2/2$. This is because
\begin{equation}\label{noo}
\nu\_{b,w}(E)\ =\ \left\{\!\!\begin{array}{cc} \frac{\ch_2(E).H - w\ch_0(E)H^3}{\ch_1^{bH}(E).H^2}
 & \text{if }\ch_1^{bH}(E).H^2\ne0, \\
+\infty & \text{if }\ch_1^{bH}(E).H^2=0 \end{array}\right.
\end{equation}
has a denominator that is linear in $b$ and numerator linear in $w$, so the walls of $\nu\_{b,w}$-instability will turn out to be \emph{linear}; see Proposition \ref{locally finite set of walls}. Note that if $\ch_i(E).H^{n-i} = 0$ for $i=0,1,2$, the slope $\nu\_{b,w}(E)$ is defined by \eqref{noo} to be $+\infty$. Since \eqref{scale} only rescales and adds a constant, it defines the same Harder-Narasimhan filtration as $N_{b,\sigma}$, so it too defines a weak stability condition on $\cA(b)$. 

\begin{Def}
Fix $w>\frac{b^2}2$. We say $E\in\cD(X)$ is $\nu\_{b,w}$-(semi)stable if and only if
\begin{itemize}
\item $E[k]\in\cA(b)$ for some $k\in\Z$, and
\item $\nu\_{b,w}(E[k])\ (\le)\ \nu\_{b,w}(F)$ for all non-trivial quotients $E[k]\to\hspace{-3mm}\to F$ in $\cA(b)$.
\end{itemize}
Here $(\le)$ denotes $<$ for stability and $\le$ for semistability.
\end{Def}

\begin{Rem}\label{heart}
Given $(b,w) \in \mathbb{R}^2$ with $w> \frac{b^2}{2}$, the argument in \cite[Propostion 5.3]{Br.stbaility} describes $\cA(b)$. It is generated by the $\nu_{b,w}$-stable two-term complexes $E = \{E^{-1} \to E^0\}$ in $\cD(X)$ satisfying the following conditions on the denominator and numerator of $\nu\_{b,w}$ \eqref{noo}:
\begin{enumerate}
    \item $\ch_1^{bH}(E).H^2 \geq 0$, and
    \item $\ch_2(E).H - w\ch_0(E)H^3 \geq 0$ if $\ch_1^{bH}(E).H^2 = 0$. 
\end{enumerate}
That is, $\cA(b)$ is the extension-closure of the set of these complexes. 
\end{Rem}

\section{Bogomolov-Gieseker inequality}
We recall the conjectural strong Bogomolov-Gieseker inequality of \cite[Conjecture 1.3.1]{BMT}, rephrased in terms of the rescaling \eqref{scale}.

\begin{Con}  \label{conjecture}
	For $\nu\_{b,w}$-semistable $E\in\cA(b)$ with $\ch_2^{bH}(E).H =\big(w -\frac{b^2}2\big)\ch_0(E)H^3$, 
$$
		\ch_3^{bH}(E)\ \leq\ \bigg(\frac{w}{3} - \frac{b^2}{6}\bigg) \ch_1^{bH}(E).H^2.
		$$ 
	\end{Con}

Although this conjecture is known not to hold for all classes on all threefolds \cite{Sc}, it is possible it always holds for objects of the classes $\ch(\iota_*L)$ that we consider. In Theorem \ref{main} we only need the conjecture in special cases, namely

\begin{enumerate}[label=\textbf{(BG\arabic*)}]
\item\label{BG1} Conjecture \ref{conjecture} holds for sheaves of class $\ch(\iota_*L)$ and stability parameters
$(-\frac n2,w)$ for any $w>\frac{n^2}4 -\frac{1}{H^3}$ for fixed $n\geq 4$.
\item\label{BG2} Conjecture \ref{conjecture} holds for both
	\end{enumerate}
\begin{itemize}
\item sheaves of class $\ch(\iota_*L)$ and stability parameters $(-\frac n2,w)$ for any $w>\frac{n^2}4 -\frac{3}{H^3}$ and fixed $n \geq 10$, and
\item torsion-free sheaves $F$ with $\ch_0(F)=1,\ \ch_1(F).H^2= 0,\ \ch_2(F).H\in\{ -1, -2\}$, and stability parameters $(b^*, w^*)$ with $b^* = \ch_2(F).H - \frac{1}{2H^3},\ w^* = (b^*)^2 + \frac{\ch_2(F).H}{H^3}$\,.
	\end{itemize}
Conjecture \ref{conjecture} is a special case of \cite[Conjecture 4.1]{BMS}, which has now been proved for
\begin{itemize}
\item $X$ is projective space $\mathbb{P}^3$ \cite{Ma}, the quadric threefold \cite{ScQ} or, more generally, any Fano threefold of Picard rank one \cite{Li.Fano},
\item $X$ an abelian threefold \cite{MP}, a Calabi-Yau threefold of abelian type \cite{BMS}, a Kummer threefold \cite{BMS}, or a product of an abelian variety and $\PP^m$ \cite{KosekiAb},
\item $X$ with nef tangent bundle \cite{Kosekinef}, and
\item $X$ is a quintic threefold and $(b,w)$ are described below \cite{Li}.
\end{itemize}

\begin{Thm} \cite[Theorem 2.8]{Li} \label{Li}
Let $X$ be a smooth quintic threefold. 
Then Conjecture \ref{conjecture} is true for $(b,w)$ satisfying 
\begin{equation}\label{in for b, w}
w\ >\ \frac{1}{2} b^2 + \frac{1}{2}\big(b - \lfloor b \rfloor\big)\big (\lfloor b \rfloor - b +1\big). 
\end{equation}
In particular \emph{\ref{BG1}} and \emph{\ref{BG2}} hold on $X$.
\end{Thm}

\begin{proof}
Using the notation $(\alpha,\beta)$ for $(w,b)$, \cite[Theorem 2.8]{Li} proves that \eqref{in for b, w} implies \cite[Conjecture 4.1]{BMS}. This gives Conjecture \ref{conjecture},
so we are left with checking that the parameters in \ref{BG1}, \ref{BG2} satisfy  \eqref{in for b, w}.

For \ref{BG1} we take $n\ge4,\ b=-\frac n2$ and $w>\frac{n^2}4 -\frac{1}{H^3}$. Then certainly $n^2>\frac8{H^3}+1$, which can be rearranged to give
\begin{equation}\label{bwn}
\frac{n^2}{8} + \frac{1}{8}\ <\ \frac{n^2}{4} - \frac{1}{H^3}\ <\ w.
\end{equation}
But since $b=-\frac n2$ we have
\begin{equation}\label{bn2}
\frac{1}{2} b^2 + \frac{1}{2}\big(b - \lfloor b \rfloor\big)\big (\lfloor b \rfloor - b +1\big)\ \le\ \frac{n^2}8+\frac18
\end{equation}
which by \eqref{bwn} gives \eqref{in for b, w}.

For \ref{BG2} we take $n\ge10,\ b=-\frac n2$ and $w>\frac{n^2}4 -\frac3{H^3}$. Then certainly 
$n^2>\frac{24}{H^3}+1$, which can be rearranged to give
$$
\frac{n^2}{8} + \frac{1}{8}\ <\ \frac{n^2}{4} - \frac3{H^3}\ <\ w.
$$
By \eqref{bn2} this gives \eqref{in for b, w}.

For the second part of \ref{BG2}, use the obvious inequality  $(2\epsilon-x)(\epsilon-x)+(\epsilon-1)x>0$ for $\epsilon\in\{1,2\}$ and $x\in(0,1)$. By rearranging this is equivalent to
$$
\frac12\left(-\epsilon-\frac x2\right)^2-\epsilon x\ >\ \frac x4\left(1-\frac x2\right).
$$
Substituting in $\epsilon=-\ch_2(F).H,\ x=\frac1{H^3}$ and $b^* = \ch_2(F).H - \frac{1}{2H^3}$ makes this
$$
\frac{(b^*)^2}2+\frac{\ch_2(F).H}{H^3}\ >\ \frac1{4H^3}\left(1-\frac1{2H^3}\right).
$$
For $w^* = (b^*)^2 + \frac{\ch_2(F).H}{H^3}$ this is 
$$
w^*-\frac{(b^*)^2}2\ >\ \frac12\left(1-\frac1{2H^3}\right)\frac1{2H^3}\ =\ \frac12\big(b^* - \lfloor b^* \rfloor\big)\big(\lfloor b^*\rfloor-b^*+1\big),
$$
i.e. the inequality \eqref{in for b, w} for $(b^*,w^*)$ as required.
\end{proof}

\section{Wall and chamber structure}
In Figure \ref{projetcion} we plot the $(b,w)$-plane simultaneously with the image of the projection map
\begin{eqnarray*}
	\Pi\colon\ K(X) \setminus \big\{E \colon \ch_0(E) = 0\big\}\! &\longrightarrow& \R^2, \\
	E &\ensuremath{\shortmid\joinrel\relbar\joinrel\rightarrow}& \!\!\bigg(\frac{\ch_1(E).H^2}{\ch_0(E)H^3}\,,\, \frac{\ch_2(E).H}{\ch_0(E)H^3}\bigg).
\end{eqnarray*}
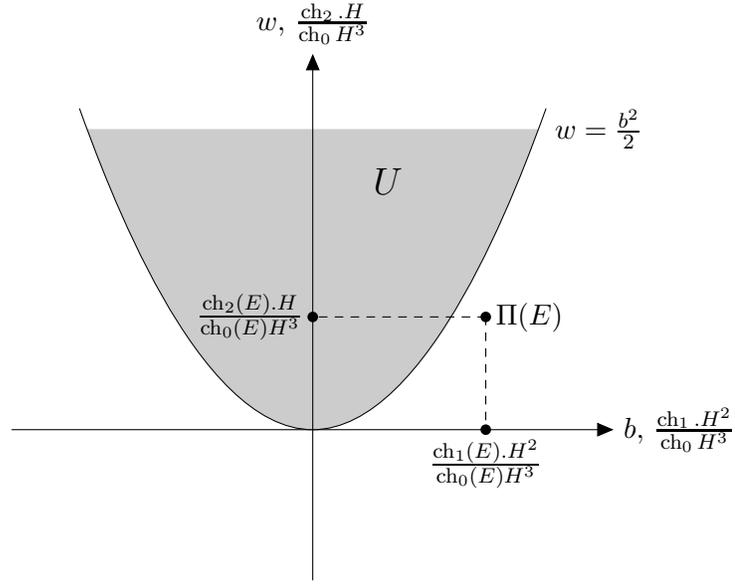
\begin{figure}[h]
	\begin{centering}
		\definecolor{zzttqq}{rgb}{0.27,0.27,0.27}
		\definecolor{qqqqff}{rgb}{0.33,0.33,0.33}
		\definecolor{uququq}{rgb}{0.25,0.25,0.25}
		\definecolor{xdxdff}{rgb}{0.66,0.66,0.66}
		
		\begin{tikzpicture}[line cap=round,line join=round,>=triangle 45,x=1.0cm,y=1.0cm]
		
		\draw[->,color=black] (-4,0) -- (4,0);
		\draw  (4, 0) node [right ] {$b,\,\frac{\ch_1.H^2}{\ch_0H^3}$};


		\fill [fill=gray!40!white] (0,0) parabola (3,4) parabola [bend at end] (-3,4) parabola [bend at end] (0,0);
		
		\draw  (0,0) parabola (3.1,4.27); 
		\draw  (0,0) parabola (-3.1,4.27); 
		\draw  (3.8 , 3.6) node [above] {$w= \frac{b^2}{2}$};
		
		

		\draw[->,color=black] (0,-2) -- (0,5);
		\draw  (0, 5) node [above ] {$w,\,\frac{\ch_2.H}{\ch_0H^3}$};

		
		\draw [dashed, color=black] (2.3,1.5) -- (2.3,0);
		\draw [dashed, color=black] (2.3, 1.5) -- (0, 1.5);
		
		\draw  (2.3, 1.5) node [right] {$\Pi(E)$};
		\draw  (1, 3) node [above] {\Large{$U$}};
		\draw  (0, 1.5) node [left] {$\frac{\ch_2(E).H}{\ch_0(E)H^3}$};
		\draw  (2.3 , 0) node [below] {$\frac{\ch_1(E).H^2}{\ch_0(E)H^3}$};
		\begin{scriptsize}
		\fill (0, 1.5) circle (2pt);
		\fill (2.3,0) circle (2pt);
		\fill (2.3,1.5) circle (2pt);
		
		
		\end{scriptsize}
		
		\end{tikzpicture}
		
		\caption{$(b,w)$-plane and the projection $\Pi(E)$}
		
		\label{projetcion}
		
	\end{centering}
\end{figure}

\noindent Note that for any weak stability condition $\nu\_{b,w}$, the pair $(b,w)$ is in the shaded open subset
\begin{equation}\label{Udef}
U \,\coloneqq \,\left\{(b,w) \in \mathbb{R}^2 \colon w > \frac{b^2}{2}  \right\}.
\end{equation}
Conversely, the image $\Pi(E)$ of $\nu\_{b,w}$-semistable objects $E$ with $\ch_0(E)\ne0$ is \emph{outside} $U$,
$$
\bigg( \frac{\ch_1(E).H^2}{\ch_0(E)H^3} \bigg)^{\!\!2} - 2\,\frac{\ch_2(E).H}{\ch_0(E)H^3}\ \geq\ 0,
$$
by the classical Bogomolov-Gieseker-type inequality of \cite[Theorem 3.5]{BMS},
\begin{equation}\label{discr}
	\Delta_H(E)\ :=\  \big(\!\ch_1(E).H^2\big)^2 -2 (\ch_0(E)H^3)(\ch_2(E).H)\ \ge\ 0,
\end{equation}
for the $H$-discriminant $\Delta_H(E)$ of a $\nu\_{b,w}$-semistable object $E$.\footnote{\cite[Theorem 3.5]{BMS} state \eqref{discr} with $\ch$ replaced by $\ch^{bH}$, but the result is still $\Delta_H(E)$. We use the stronger Bogomolov inequality $\ch_1(E)^2.H-2\ch_0(E)(\ch_2(E).H)\ge0$ for $\mu\_H$-semistable sheaves in \eqref{condition 2}.}

\begin{Prop}[\textbf{Wall and chamber structure}]\label{locally finite set of walls}
	Fix an object $E \in \mathcal{D}(X)$ such that the vector $\left(\ch_0(E), \ch_1(E).H^2, \ch_2(E).H\right) \neq 0$ is non-zero. There exists a set of lines $\{\ell_i\}_{i \in I}$ in $\mathbb{R}^2$ such that the segments $\ell_i\cap U$ (called ``\emph{walls}") are locally finite and satisfy 
	\begin{itemize*}
	    \item[\emph{(a)}] Any line $\ell_i$ passes through the point $\Pi(E)$ if $\ch_0(E) \neq 0$, or has fixed slope $\frac{\ch_2(E).H}{\ch_1(E).H^2}$ if $\ch_0(E) = 0$. 
	    
		\item[\emph{(b)}] The $\nu\_{b,w}$-(semi)stability of $E$ is unchanged as $(b,w)$ varies within any connected component (called a ``\emph{chamber}") of $U \setminus \bigcup_{i \in I}\ell_i$.
		
		\item[\emph{(c)}] For any wall $\ell_i\cap U$ there is $k_i \in \mathbb{Z}$ and a map $f\colon F\to E[k_i]$ in $\cD(X)$ such that
\begin{itemize}
\item for any $(b,w) \in \ell_i \cap U$, the objects $E[k_i],\,F$ lies in the heart $\cA(b)$,
\item $E[k_i]$ is $\nu\_{b,w}$-semistable with $\nu\_{b,w}(E)=\nu\_{b,w}(F)=\,\mathrm{slope}\,(\ell_i)$ constant on $\ell_i \cap U$, and
\item $f$ is an injection $F\subset E[k_i]$ in $\cA(b)$ which strictly destabilises $E[k_i]$ for $(b,w)$ in one of the two chambers adjacent to the wall $\ell_i$.

\end{itemize} 
	\end{itemize*} 
	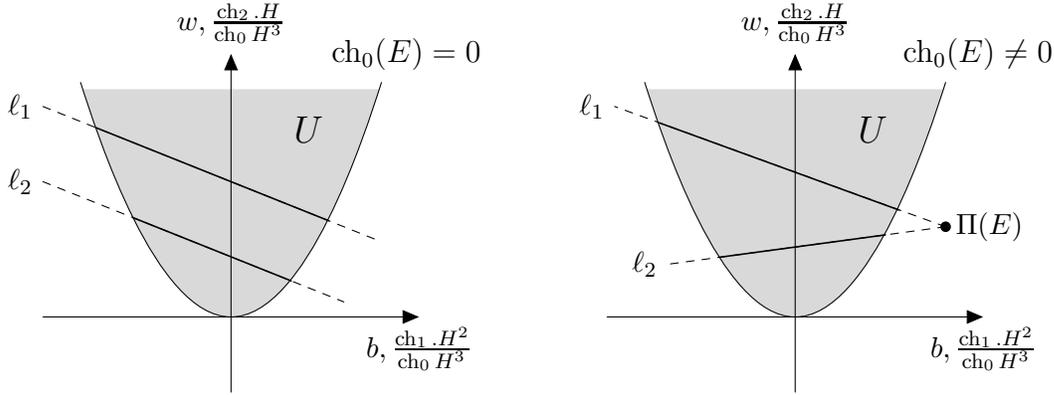
\begin{figure} [h]
	\begin{centering}
		
		\begin{tikzpicture}[line cap=round,line join=round,>=triangle 45,x=1.0cm,y=1.0cm]
	
		\draw[->,color=black] (-10.5,0) -- (-5.5,0);
		\draw[->,color=black] (-3,0) -- (2,0);
		
		\fill [fill=gray!30!white] (-0.5,0) parabola (1.47, 3.03) parabola [bend at end] (-2.47,3.03) parabola [bend at end] (-0.5,0);
		
		\fill [fill=gray!30!white] (-8,0) parabola (-6.03, 3.03) parabola [bend at end] (-9.97,3.03) parabola [bend at end] (-8,0);

		\draw[->,color=black] (-8,-1) -- (-8,3.5);
		\draw[->,color=black] (-0.5,-1) -- (-0.5,3.5);

		\draw [] (-0.5,0) parabola (1.5,3.12); 
		\draw [] (-0.5,0) parabola (-2.5,3.12); 
		\draw [] (-8,0) parabola (-10,3.12); 
		\draw [] (-8,0) parabola (-6,3.12);

		\draw[color=black, dashed] (-10.5,2.8) -- (-6,1);
		\draw[color=black, dashed] (-10.5,1.8) -- (-6.5,0.2);

       \draw[color=black,semithick] (-9.8,2.52) -- (-6.7,1.28);
       \draw[color=black,semithick] (-9.3,1.32) -- (-7.2,.48);
       
		\draw (-10.5,1.8) node [left] {$\ell_2$};
		\draw (-10.5,2.8) node [left] {$\ell_1$};

		\draw (.8,3.5) node [right] {\large{$\ch_0(E) \neq 0$}};
		\draw (-6.8,3.5) node [right] {\large{$\ch_0(E) = 0$}};
		\draw (-5.5,0) node [below] {$b, \frac{\ch_1.H^2}{\ch_0H^3}$};
		\draw (-8,3.5) node [above] {$w, \frac{\ch_2.H}{\ch_0H^3}$};
		
		\draw (2,0) node [below] {$b, \frac{\ch_1.H^2}{\ch_0H^3}$};
		\draw (-0.5,3.5) node [above] {$w, \frac{\ch_2.H}{\ch_0H^3}$};
	
		\draw (-7.3,2.5) node [right] {\Large{$U$}};
		\draw (0.2,2.5) node [right] {\Large{$U$}};

		\draw (1.5, 1.2) node [right] {$\Pi(E)$};
		
		\draw[color=black, dashed] (1.5, 1.2) -- (-2.9,2.8);
		\draw[color=black, dashed] (1.5, 1.2) -- (-2.2, .7);
		
		\draw (-2.9,2.8) node [left] {$\ell_1$};
		\draw (-2.2, .7) node [left] {$\ell_2$};

		\draw[color=black, semithick] (-2.3 ,2.58) -- (.88,1.423);
		\draw[color=black, semithick] (-1.5,.795) -- (0.7,1.092);

		\begin{scriptsize}
	
%
%
%

		\fill [color=black] (1.5,1.2) circle (2pt);
		
%

		
%
%
%
%
%
%
		

		\end{scriptsize}
		
		\end{tikzpicture}
		
		\caption{The line segments $\ell_i \cap U$ are walls for $E$.}
		
		\label{wall.figure}
		
	\end{centering}
	
\end{figure}
\end{Prop}
\begin{proof}
	For $E \in \mathcal{D}(X)$ the existence of a locally finite set of walls in the $(b,w)$ plane follows from the arguments in \cite[Proposition 9.3]{Br} or \cite[Proposition 12.5]{BMS}. 
	
	Suppose that $E$ is $\nu\_{b,w}$-strictly semistable. Then there is a $k\in\Z$ such that $E[k]\in\cA(b)$ and a $\nu_{b,w}$-stable destabilising object $F\subset E[k]$ in $\cA(b)$. The condition that $\nu\_{b,w}(E[k])=\nu\_{b,w}(F)$ is
\begin{equation}\label{one}
\frac{w-\frac{\ch_2(E[k]).H}{\ch_0(E[k])H^3}}{b-\frac{\ch_1(E[k]).H^2}{\ch_0(E[k])H^3}}\ =\ \frac{w-\frac{\ch_2(F).H}{\ch_0(F)H^3}}{b-\frac{\ch_1(F).H^2}{\ch_0(F)H^3}} \quad\mathrm{if}\ \,\ch_0(E[k])\ne0\ne\ch_0(F),
\end{equation}
or
\begin{equation}\label{two}
\frac{w-\frac{\ch_2(E[k]).H}{\ch_0(E[k])H^3}}{b-\frac{\ch_1(E[k]).H^2}{\ch_0(E[k])H^3}}\ =\ \frac{\ch_2(F).H}{\ch_1(F).H^2} \quad\mathrm{if}\ \,\ch_0(E[k])\ne0=\ch_0(F),
\end{equation}
or
\begin{equation}\label{three}
\frac{\ch_2(E[k]).H}{\ch_1(E[k]).H^2}\ =\ \frac{w-\frac{\ch_2(F).H}{\ch_0(F)H^3}}{b-\frac{\ch_1(F).H^2}{\ch_0(F)H^3}} \quad\mathrm{if}\ \,\ch_0(E[k])=0\ne\ch_0(F).
\end{equation}
As we move through the $(b,w)$ plane, \eqref{one} is the equation of the straight line joining $\Pi(E)$ and $\Pi(F)$, \eqref{two} is the straight line though $\Pi(E)$ of slope $\frac{\ch_2(F).H}{\ch_1(F).H^2}$, and \eqref{three} is 
the line through $\Pi(F)$ of slope $\frac{\ch_2(E[k]).H}{\ch_1(E[k]).H^2}$. In each case the slopes of $E[k]$ and $F$ are constant on the wall, and satisfy strict (and opposite) inequalities on the two sides of the wall. This explains the shape of the walls of instability.

If $\ch_0(E[k])=0=\ch_0(F)$ we do not get a wall since both slopes remain constant as we move throughout the whole of $U$ in the $(b,w)$ plane.

Finally, if we move along a wall, the $\nu\_{b,w}$-slopes of all the Jordan-H\"older factors of $E[k]$ coincide and remain constant. So long as they're finite, Remark \ref{heart} implies that the Jordan-H\"older factors remain in the heart $\cA(b)$, and so $E[k]$ does too. If they're infinite the wall is vertical, and the category $\cA(b)$ is constant, so the conclusion is the same.
\end{proof}
	


%
%

\section{Large volume limit}	
As usual we consider a line bundle $L$ on $D \in |\mathcal{O}(n)|$ such that $L.H = 0$. The Chern character of its push-forward is  
\begin{equation}\label{ch}
\ch(\iota_*L)\ =\ \left(0,\,nH,\,\iota_*(c_1(L))-\frac{n^2}2H^2,\,\frac12L^2+\frac{n^3}6H^3\right).
\end{equation}
To move through the space $U$ \eqref{Udef} of weak stability conditions, we begin in the large volume region $w\gg0$. We use the fact that $L$ is slope stable on $D$ since it has no proper saturated subsheaves when $D$ is irreducible. (The results of this paper also hold for reducible $D$ if we \emph{assume} that $\iota_*L$ is slope semistable.)

	\begin{Lem} \label{large volume} 
		The sheaf $\iota_*L$ is $\nu\_{b,w}$-semistable for any $b \in\R$ and $w \gg 0$.	
	\end{Lem}
	\begin{proof}
We sketch the proof, which is very similar to \cite[Proposition 14.2]{Br}. The key point is that a sheaf $\iota_*E$ pushed forward from $D$ has rank 0 so its $\nu\_{b,w}$-slope \eqref{noo},
		\begin{equation}\label{indep}
		\nu\_{b,w}(\iota_*E)\ =\ \frac{\ch_2(\iota_*E).H}{\ch_1(\iota_*E).H^2}\ =\ \frac{\ch_1(E).H}{\ch_0(E)H^2}-\frac n2\ =\ \mu\_H(E)-\frac n2\,,
		\end{equation}
is independent of $(b,w) \in \mathbb{R}^2$ and essentially reduces to the ordinary slope of $E$ on $D$. Here the intersections take place on $X$ in the second term and on $D$ in the third term. (On reducible $D$ the denominator $\ch_0(E)H^2$ would be replaced by the leading coefficient of the Hilbert polynomial of $E$.)
		
		Fix a real number $b \in \mathbb{R}$. The sheaf $\iota_* L$ is in the heart $\mathcal{A}(b)$. Fix a subobject $E_1$ of $\iota_* L$ in $\mathcal{A}(b)$ with quotient $E_2$. Then the ordinary cohomology sheaves $\cH^i$ of these objects sit in a long exact sequence 
		\begin{equation*}
		0 \To \cH^{-1}(E_2) \To \cH^0(E_1) \To \iota_*L \To \cH^0(E_2) \To 0.
		\end{equation*} 
In particular $E_1$ is a sheaf. Suppose first that rank\,$(E_1)\ne0$. Since $E_1\in\cA(b)$ we know $\mu_H^-(E_1)>b\Rightarrow\mu\_H(E_1)>b\Rightarrow \ch_1^{bH}(E_1).H^2>0$. By \eqref{noo} therefore, $+\infty>\nu\_{b,w}(E_1)\to-\infty$ as $w\to\infty$, so $E_1$ does not destabilise $\iota_*L$ for $w\gg0$. As in \cite[Proposition 14.2]{Br} one can in fact make the bound on $w$ (so that $E_1$ does not destabilise) uniform in $E_1$.

If rank$\,(E_1)=0$ then $\cH^{-1}(E_2) = 0$ because $E_2\in\cA(b)$ implies that $\cH^{-1}(E_2)$ is a torsion-free sheaf. Therefore $E_1$ is a subsheaf of $\iota_* L$, which by \eqref{indep} and the slope semistability of $L$ cannot strictly $\nu\_{b,w}$-destabilise $\iota_* L$.  
%
%
%
\end{proof}


\section{The first wall}
From now on we work in one of the situations
\begin{itemize*}
	\item[(i)] suppose \ref{BG1} holds, $n \geq 4$ and $L^2 \ge \big\lfloor\frac{-2n}{3}\big\rfloor+1$, or
	\item[(ii)] suppose \ref{BG2} holds, $n \geq 10$ and $L^2 \ge -2n+5$.
\end{itemize*}
Then moving in the space $U$ of weak stability conditions we will try to show that $c_1(L)$ is a torsion class in $H^2(D,\Z)$. This will prove Theorem \ref{main}.

By Proposition \ref{locally finite set of walls} the walls of instability for $\iota_*L$ are all lines of slope $-\frac n2$ in the $(b,w)$ plane; see Figure \ref{walls}. The lowest such line which intersects $\overline U$ is $w=-\frac n2b-\frac{n^2}8$, which is tangent to $\partial U$ at $(-\frac n2,\frac{n^2}8)$. Therefore the vertical line
\begin{equation}\label{bb0}
b\ \equiv\ b_0\ :=\ -\frac n2
\end{equation}
intersects all the possible walls of instability of $\iota_*L$. We will move down this vertical line from the large volume region $w\gg0$.

By \eqref{ch}, $\ch_2^{bH}(\iota_* L).H=0=\ch_0(\iota_* L)$ on the line $b=b_0$, so we can apply the Bogomolov-Gieseker Conjecture \ref{conjecture} for stability parameters  $\left(-\frac n2\,, w \right)$. 
That is,  if $\iota_*L$ is $\nu\_{b_0,w}$-semistable then
$$
\ch_3^{b_0H}(\iota_* L)\ \leq\ \bigg(\frac{w}{3} - \frac{b_0^2}{6}\bigg)\ch_1^{b_0H}(\iota_* L).H^2.
$$
Using \eqref{ch} and rearranging gives
\begin{equation}\label{final bound}
    w\ \ge\ w_f\ :=\ \dfrac{n^2}{4} + \dfrac{3L^2}{2nH^3}\,.
\end{equation}

Note that case (i) gives $w_f>\frac{n^2}4-\frac1{H^3}$, while case (ii) gives $w_f > \frac{n^2}4-\frac{3n-6}{nH^3} > \frac{n^2}4-\frac{3}{H^3}$. In both cases then, $w_f > \frac{b_0^2}{2} = \frac{n^2}{8}$, so $(b_0,w_f)$ lies inside $U$.

Therefore, when we move down the line $b=-\frac n2$, we find there is a point $w_0\ge w_f$ where $\iota_*L$ is first destabilised. We next show that in fact $w_0\in\big[w_f , \frac{n^2}4\big]$.
\vspace{.2 cm}



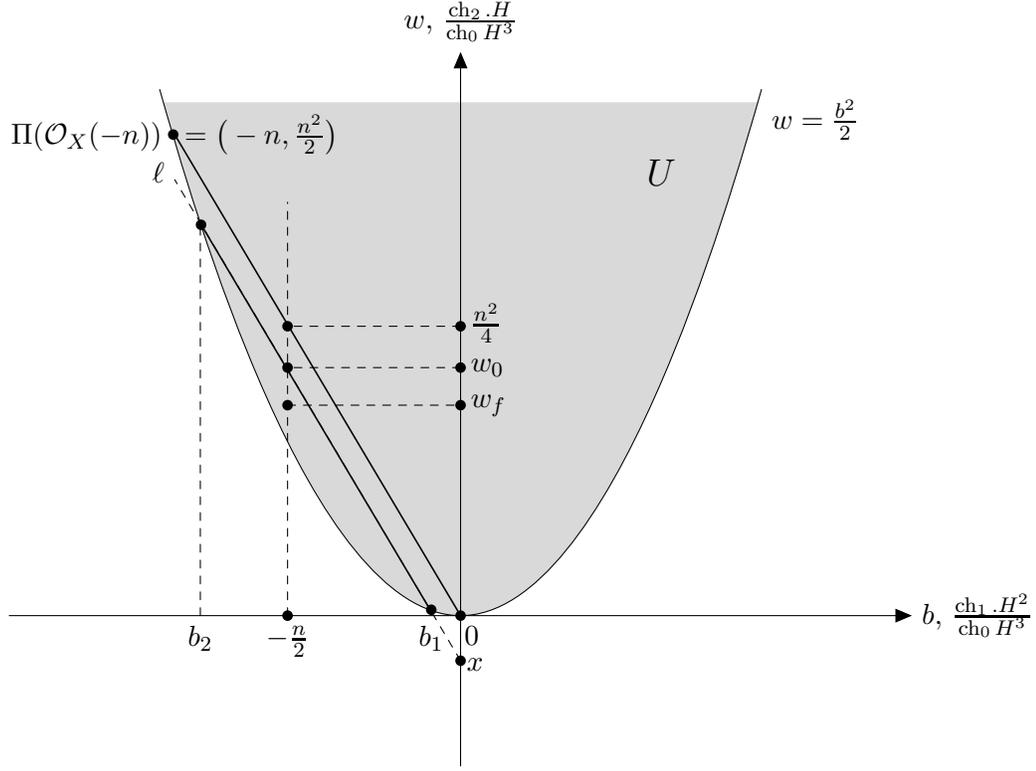
\begin{figure}[h]
	\begin{centering}
		\definecolor{zzttqq}{rgb}{0.27,0.27,0.27}
		\definecolor{qqqqff}{rgb}{0.33,0.33,0.33}
		\definecolor{uququq}{rgb}{0.25,0.25,0.25}
		\definecolor{xdxdff}{rgb}{0.66,0.66,0.66}
		
		\begin{tikzpicture}[line cap=round,line join=round,>=triangle 45,x=1.0cm,y=1.0cm]
		
		\draw[->,color=black] (-6,0) -- (6,0);
		\draw  (6, 0) node [right ] {$b,\,\frac{\ch_1.H^2}{\ch_0H^3}$};

		\fill [fill=gray!30!white] (0,0) parabola (3.95, 6.83) parabola [bend at end] (-3.95, 6.83) parabola [bend at end] (0,0);
		
		\draw  (0,0) parabola (4,7); 
		\draw  (0,0) parabola (-4,7); 
		\draw  (4 , 6.6) node [right] {$w=\frac{b^2}{2}$};
		
		

		\draw[->,color=black] (0,-2) -- (0,7.5);
		\draw  (0, 7.5) node [above ] {$w,\,\frac{\ch_2.H}{\ch_0H^3}$};
		
		
		\draw[dashed,color=black] (-2.3,0) -- (-2.3,5.5);
		
		\draw[dashed,color=black] (0,3.85) -- (-2.3,3.85);
		
		\draw[dashed,color=black] (0,3.3) -- (-2.3,3.3);
		
		\draw [color=black, semithick] (0,0) -- (-3.8,6.4);
		
		\draw [color=black, semithick] (-.4, .07)  --  (-3.46,  5.22);
			\draw [color=black, dashed] (-.39, .07)  --  (-.39, 0);
			\draw [color=black, dashed] (-3.46, 0)  --  (-3.46,  5.2);
		\draw [color=black, dashed] (0,-.6) -- (-3.8,5.8);

		\draw [color=black, dashed] (-2.3, 2.8)  --  (0,  2.8);
		
		\draw  (-3.8, 6.8) node [below ] {$\Pi(\mathcal{O}_X(-n))\ \,=\big(-n,\frac{n^2}2\big)$};
		\draw  (0, 3.9) node [right ] {$\frac{n^2}{4}$};
		\draw  (0, 3.3) node [right ] {$w_0$};
		\draw  (0, 2.8) node [right ] {$w_f$};
		
		\draw  (-2.3, 0) node [below ] {$-\frac{n}{2}$};
		\draw  (.15, 0) node [below] {0};
		\draw  (-3.8,5.9) node [left ] {$\ell$};
		\draw  (3,5.9) node [left ] {\Large{$U$}};
		
		\draw  (-3.46,0) node [below] {$b_2$};
		\draw  (-.39, 0) node [below ] {$b_1$};
			\draw  (-.05, -.65) node [right] {$x$};
		
		\begin{scriptsize}
		\fill (-2.3, 2.8) circle (2pt);
		\fill (0, 2.8) circle (2pt);
				
		\fill (0,0) circle (2pt);
		
		\fill (0, -.6) circle (2pt);
		\fill (0, 3.85) circle (2pt);
		
		\fill (-3.82,6.4) circle (2pt);
		\fill (-2.3,0) circle (2pt);
		\fill (-2.3,3.85) circle (2pt);

		\fill (-2.3,3.3) circle (2pt);
		\fill (0,3.3) circle (2pt);

    	\fill (-.39, .078) circle (2pt);
		\fill (-3.45,  5.2) circle (2pt);

		\end{scriptsize}
		
		\end{tikzpicture}
		
		\caption{Walls for $\iota_*L$}
		
		\label{walls}
		
	\end{centering}
\end{figure}

	\begin{Prop}\label{prop.1}
There is a wall of slope $-\frac n2$ for $\iota_* L$ that bounds the large volume chamber $w\gg0$. It passes through a point $(b_0, w_0)$, where $w_0 \in\big[w_f , \frac{n^2}{4}\big]$.

In the destabilising sequence $F_1 \hookrightarrow \iota_* L \twoheadrightarrow F_2$ in $\cA(b_0)$, we have $\dim\mathrm{supp}\,\cH^0(F_2)\le1$, the 
object $F_1$ is a rank one sheaf with $\ch_1(F_1).H^2= 0$ and, in cases \emph{(i),\,(ii),}
\begin{itemize*}
	\item[\emph{(i)}] $\ch_2(F_1).H = 0$,
	\item[\emph{(ii)}] $\ch_2(F_1).H\in\{0,-1,-2\}$.
\end{itemize*}
 
	\end{Prop} 

	\begin{proof}
By Proposition \ref{locally finite set of walls} and \eqref{final bound}, $\iota_*L$ is $\nu\_{b_0,w_0}$-destabilised by a sequence $F_1 \hookrightarrow \iota_* L \twoheadrightarrow F_2$ in $\cA(b_0)$ for $b_0=-\frac n2$ and some $w_0\ge w_f$. The corresponding wall is denoted by $\ell$ in Figure \ref{walls}. It has equation $w = -\frac{n}{2}b + x$, where
		\begin{equation*}
		x\ =\ w_0-\frac{n^2}4\ \geq\  w_f-\frac{n^2}{4}\ =\  \frac{3L^2}{2nH^3} 
		\end{equation*} 
satisfies
\begin{equation}\label{x>}
x\ >\ \left\{\!\!\begin{array}{cc}-\frac{1}{H^3} & \mathrm{in\ case\ (i)},\vspace{1mm} \\
-\frac{3}{H^3} & \mathrm{in\ case\ (ii).}\end{array}\right.
\end{equation}
		
Let $b_2 < b_1$ be the values of $b$ at the intersection points of $\ell$ and the boundary  $w =\frac{b^2}{2}$ of $U$, 
		\begin{equation*}
		b_1\ =\ \sqrt{{\frac{n^2}{4}}+2x}-\frac n2\,, \qquad
		b_2\ =\ -\sqrt{{\frac{n^2}{4}}+2x}-\frac n2\,.
		\end{equation*}	
		We claim that 
		\begin{equation}\label{b1b2}
		b_1\ >\ -\frac{1}{2H^3} \;\;\; \text{and} \;\;\; 	b_2+n\ <\ \dfrac{1}{2H^3}\,.
		\end{equation}
Both are equivalent to $\sqrt{{\frac{n^2}{4}}+2x} > \frac{n}{2} -\frac{1}{2H^3}$, and therefore to $2x > \frac{1}{4(H^3)^2} - \frac{n}{2H^3}$. Since $x\ge \frac{3L^2}{2nH^3}$ it is sufficient to show
$$
\frac{3L^2}n\ \ge\ \frac{1}{4H^3} - \frac n2\,.
$$
For (i) this follows from $L^2\ge\big\lfloor-\frac{2n}3\big\rfloor+1\ge-\frac{2n}3+\frac13$ and the inequality $-2+\frac1n>-\frac n2+\frac1{4H^3}$ that holds for $n\ge 4$. For (ii) it follows from $L^2\ge-2n+5$ and the inequality $\frac n2>6-\frac{15}n+\frac1{4H^3}$ that holds for all $n\ge10$.
	
		
		Taking cohomology from the destabilising sequence $F_1 \hookrightarrow \iota_* L \twoheadrightarrow F_2$ gives the long exact sequence of coherent sheaves 
		\begin{equation}\label{long exact}
		0 \To \cH^{-1}(F_2) \To\cH^{0}(F_1) \To \iota_* L \To \cH^0(F_2) \To 0.
		\end{equation}
		In particular, the destabilising subobject $F_1$ is a coherent sheaf. As we saw in the proof of Proposition \ref{locally finite set of walls}, if it had rank 0 then its slope would be constant throughout $U$, like that of $\iota_*L$, so we would not have a wall. Thus $\ch_0(F_1)>0$ so \eqref{long exact} gives
		 $$
		 \ch_0(\cH^{-1}(F_2))\ =\ \ch_0(F_1)\ >\ 0.
		 $$
As in Proposition \ref{locally finite set of walls}, $\Pi(F_1)$ and $\Pi(F_2)$ lie on the line $\ell$. All along $\ell\cap U$ (i.e. for $b \in (b_2, b_1)$) the objects $F_1$ and $F_2$ lie in the heart $\mathcal{A}(b)$ and (semi)destabilise $\iota_*L$. Therefore by the definition \eqref{Abdef} of $\cA(b)$ and the inequalities \eqref{b1b2},
		\begin{equation}\label{two conditions}
		\mu_H^{+}(\cH^{-1}(F_2))\ \leq\ b_2\ <\ -n + \frac{1}{2H^3} 
		\;\;\; \text{and} \;\;\; \mu_H^{-}(F_1)\ \geq\ b_1\ >\ \dfrac{-1}{2H^3}\,.
		\end{equation}  
Thus dividing $\left(\ch_1(\iota_*L)-\ch_1(\cH^0(F_2))\right).H^2=\left(\ch_1(F_1)-\ch_1(\cH^{-1}(F_2))\right).H^2$ by $\ch_0(F_1)H^3$ gives
		\begin{align}\label{in.1}
\frac n{\ch_0(F_1)}-\dfrac{\ch_1(\cH^0(F_2)).H^2}{\ch_0(F_1)H^3}\ &=\ \mu\_H(F_1) -\mu\_H(\cH^{-1}(F_2)) \\ \nonumber
&\ge\ \mu_H^{-}(F_1) - \mu_H^{+}(\cH^{-1}(F_2))\ \geq\ b_1 - b_2\ >\ n- \frac{1}{H^3}\,.
		\end{align}
		Since $\cH^0(F_2)$ has rank zero, $\ch_1(\cH^0(F_2)).H^2 \geq 0$ so $\frac n{\ch_0(F_1)}\ge$ \eqref{in.1}. Thus the inequalities imply $\ch_0(F_1) =1$ and $\ch_1(\cH^0(F_2)).H^2=0$. In particular, $\cH^0(F_2)$ is supported in dimension $\le1$.\medskip
				
		Hence $\mu\_H(F_1) = \frac{\ch_1(F_1).H^2}{H^3}$ is an integer multiple of $\frac{1}{H^3}$, so the inequality \eqref{two conditions} implies that $\mu\_H(F_1) \geq 0$.
		Similarly \eqref{two conditions} gives $\mu\_H(\cH^{-1}(F_2)) \leq -n$ while \eqref{in.1} gives $\mu\_H(\cH^{-1}(F_2)) \geq -n$. The upshot is that $\mu\_H(F_1) = 0$ and $\mu\_H(\cH^{-1}(F_2)) = -n$. Hence $\ch_1(F_1).H^2 = 0$ and $\Pi(F_1)$ lies on the $w$-axis. But it also lies on the wall $\ell$ given by $w = \frac{-n}{2} b + x$, so
	  \begin{equation}\label{x=}
	  x\ =\ \frac{\ch_2(F_1).H}{\ch_0(F_1)H^3}\ =\ \frac{\ch_2(F_1).H}{H^3}\,. 
	  \end{equation} 
	  Since the sheaf $F_1$ is $\nu_{b_0,w_0}$-semistable, $\Pi(F_1)$ lies outside $U$ by \eqref{discr}. Thus $x \leq 0$ which is $w_0 \leq \frac{n^2}{4}$, as claimed. Combining this with \eqref{x>} and \eqref{x=} gives, finally,
\[	  
0\ \ge\ \ch_2(F_1).H\ >\ 
\left\{\!\!\begin{array}{cc}-1 & \mathrm{in\ case\ (i)}, \\
-3 & \mathrm{in\ case\ (ii).}\end{array}\right.\qedhere
\]
	\end{proof}
	
   \begin{Prop}\label{prop.case ii}
   		Under the assumptions of Proposition \ref{prop.1},
 the destabilising subobject $F_1$ of $\iota_* L$ satisfies $\ch_2(F_1).H = 0$. That is, $x=0,\,w_0=\frac{n^2}4$, and the wall bounding the large volume chamber is the line of slope $-\frac n2$ through the origin.
   \end{Prop}
 
	Proposition \ref{prop.1} proves this in case (i). We will prove Proposition \ref{prop.case ii} in case (ii) in Section \ref{last section} by applying the Bogomolov-Gieseker conjecture \ref{conjecture} to $F_1$ and $F_2$. This gives upper bounds for $\ch_3(F_1)$ and $\ch_3(F_2)$ respectively. In turn the latter gives a lower bound for $\ch_3(F_1)$. If we work only at $(b_0,w_0)$, as in \cite{TodaBG}, the bounds are not optimal, but by working at more general points of the  $(b,w)$-plane we get stronger bounds which together force $\ch_2(F_1).H = 0$.  

	\begin{Lem}\label{3}
Under the assumptions of Proposition \ref{prop.1}, $\dim\mathrm{supp}\,\cH^0(F_2)=0$ and
$$
\ch_1(\cH^{-1}(F_2))\ =\ -nH \quad\mathrm{in\ } H^2(X,\Q).
$$
	\end{Lem}
\begin{proof}
By Proposition \ref{prop.1}, $F_2$ has rank 1 and lies in $\cA(b_0)$ \eqref{Abdef}, so $\cH^{-1}(F_2)$ is a torsion-free rank one sheaf. Therefore it is $\mu\_H$-semistable and the classical Bogomolov inequality says 
	\begin{equation}\label{condition 2}
	\ch_1(\cH^{-1}(F_2))^2.H -2\ch_2(\cH^{-1}(F_2)).H\ \geq\ 0. 
	\end{equation}
	From the exact sequence \eqref{long exact} we calculate $\ch_i(\cH^{-1}(F_2))=\ch_i(F_1)-\ch_i(\iota_*L)+\ch_i(\cH^0(F_2))$. 	Taking $i=2$ and intersecting with $H$, Proposition \ref{prop.case ii} kills the first term while \eqref{ch} and $L.H=0$ calculate the second, yielding
	\begin{equation}\label{ch2}
	\ch_2(\cH^{-1}(F_2)).H\ =\ \frac{n^2H^3}{2} + \ch_2(\cH^0(F_2)).H.
	\end{equation}
Taking $i=1$ and intersecting with $H^2$, Proposition \ref{prop.1} kills the first and third terms, giving
	$$
	\ch_1(\cH^{-1}(F_2)).H^2\ =\ -nH^3.
	$$
So by the Hodge index theorem
\begin{equation}\label{ch11}
n^2H^3\ =\ \frac{\big(\!\ch_1(\cH^{-1}(F_2)).H^2\big)^2}{H^3}\ \ge\ \ch_1(\cH^{-1}(F_2))^2.H,
\end{equation}
\emph{with equality if and only if $\ch_1(\cH^{-1}(F_2))$ is a multiple of $H$ in $H^2(X,\Q)$.}

Combining \eqref{condition 2}, \eqref{ch2} and \eqref{ch11} gives
\begin{equation}\label{<>0}
-2\ch_2(\cH^0(F_2)).H\ \ge\ 0.
\end{equation}
But Proposition \ref{prop.1} also showed that $\cH^0(F_2)$ is supported in dimension $\le1$, so \eqref{<>0} shows it must have 0-dimensional support and (\ref{<>0}, \ref{ch11}) are equalities. Thus $\ch_1(\cH^{-1}(F_2))$ is a multiple of $H$ in $H^2(X,\Q)$.

To determine the multiple we calculate from the sequence \eqref{long exact} that $\ch_1(\cH^{-1}(F_2)).H^2=\ch_1(F_1).H^2-\ch_1(\iota_*L).H^2$. The former is zero by Proposition \ref{prop.1} and the second is $nH^3$.
	\end{proof}
	
So $\cH^0(F_2)$ is supported in dimension 0 and is a quotient of $\iota_*L$ by \eqref{long exact}. Thus there is a 0-dimensional subscheme $Z\subset D$ with ideal sheaf $I_Z$ \emph{on $D$} such that \eqref{long exact} simplifies to
		\begin{equation}\label{long exact2}
0 \To \cH^{-1}(F_2) \To F_1 \To \iota_*(L\otimes I_Z) \To 0,\end{equation}
where $\cH^{-1}(F_2)$ and $F_1$ are rank 1 torsion free sheaves. 
By Lemma \ref{3} there is a dim $\le1$ subscheme $C\subset X$ such that
\begin{equation}\label{TC}
\cH^{-1}(F_2)\ \cong\ T(-n)\otimes I_C
\end{equation}
for some line bundle $T$ with $c_1(T)=0\in H^2(X,\Q)$. Rotating the exact triangle \eqref{long exact2}, we get a short exact sequence in $\cA(b_0)$:
		\begin{equation}\label{surj1}
		0 \To F_1 \To \iota_*(L\otimes I_Z) \To T(-n)\otimes I_C[1] \To 0.
		\end{equation}
In fact any rank zero sheaf such as $\iota_*(L\otimes I_Z)$ lies in the heart $\mathcal{A}(b_0)$. Since $T(-n)$ is a line bundle, it is a $\mu\_H$-semistable sheaf of the same slope as $\cH^{-1}(F_2)$, and thus its shift by $[1]$ lies in $\mathcal{A}(b_0)$ because $F_2$ does. By the same reasoning,
		\begin{equation}\label{sec}
		0 \To T(-n)\otimes\cO_C \rightarrow T(-n)\otimes I_C[1] \To T(-n)[1] \To 0
		\end{equation}
		is also a short exact sequence in $\mathcal{A}(b_0)$.

%

\section{Proof of main Theorem}
We are now ready to prove Theorem \ref{main}. We compose the $\cA(b_0)$-surjections (the third arrows) of \eqref{surj1} and \eqref{sec} to give
$$
\iota_*(L\otimes I_Z)\To T(-n)[1].
$$ 
Since this is a surjection in $\cA(b_0)$, it is a nonzero element of 
\begin{equation}\label{noZ}
\Ext^1(\iota_*(L\otimes I_Z),T(-n))\ \cong\ 
\Ext^1(\iota_*L,T(-n))\ \cong\ \Hom(L,T|_D).
\end{equation}
(The first isomorphism follows from $\Ext^{<3}(\cO_Z,T(-n))=0$, by $\dim Z=0$, and the second from relative Serre duality for $\iota$.) Thus $L^* \otimes T|_D$ is effective. Since $L.H=0$ this implies $L=T|_D$. In particular, $c_1(L)=0$ in $H^2(D,\Q). \hfill\square$

\begin{Rem}
In fact, calculating $\ch_2(F_1).H$ from \eqref{long exact2} and \eqref{TC} gives $-H.C$, which by Proposition \ref{prop.case ii} is zero. Therefore both $C$ and $Z$ are 0-dimensional and the $\nu\_{b,w}$ slopes of $T(-n)\otimes I_C$ and $\iota_*(L\otimes I_Z)$ are the same as those of $T(-n)$ and $\iota_*L$ respectively. Thus the map $\iota_*L\to T(-n)[1]$ produced in \eqref{noZ} also destabilises in $\cA(b)$ on the first wall. That is, 
$$
0\To\cO(-n)\To\cO\To\cO_D\To0
$$
-- tensored with $T$ and rotated -- gives the destabilising short exact sequence in $\cA(b)$.
\end{Rem}
	\section{Destabilising objects in case $\ii$}\label{last section}
What remains is to prove Proposition \ref{prop.case ii} in case (ii). So we assume \ref{BG2} holds, $n\ge10$ and $L^2\ge-2n+5$. 
By Proposition \ref{prop.1}, in $\cA(b)$ there is a destabilising sequence $F_1 \hookrightarrow \iota_*L \twoheadrightarrow F_2$  for $\iota_*L$ along the wall $\ell$ with equation
	$$
	w\ =\ -\frac{n}{2} b + \frac{\ch_2(F_1).H}{H^3}\,.
	$$ 
Moreover $\rk F_1=1=-\rk F_2$, and, by Proposition \ref{prop.1}, 
   \begin{equation*}
\ch_1(F_1).H^2 = 0 \quad\text{and}\quad \ch_2(F_1).H\in\{0,-1,-2\}. 
   \end{equation*}
   We will assume that $\ch_2(F_1).H\ne0$ and apply the Bogomolov-Gieseker inequality to $F_1$ and $F_2$ to get a contradiction. \medskip
  
It will be convenient to work with $b=b_1:= -\frac{1}{H^3}$ because then, by \eqref{noo},
$$
\nu\_{b_1,w}(E)\ =\ \frac{\ch_2(E).H - w\ch_0(E)H^3}{\ch_1(E).H^2+\ch_0(E)}
$$
has a denominator $D_1(E):=\ch_1(E).H^2+\ch_0(E)$ which
\begin{itemize}
\item is integral and $\ge0$ for $E\in\cA(b_1)$,
\item is additive on K-theory classes: $D_1(E_1+E_2)=D_1(E_1)+D_1(E_2)$, and
\item takes the minimal nonzero value $1$ on $F_1$.
\end{itemize}
This means that in $\cA(b_1)$ the object $F_1$ can only be destabilised by objects with denominator $D_1=0$.\footnote{This argument is familiar from the analogous fact that rank 1 sheaves can only be destabilised by rank 0 torsion sheaves when working with slope (for which the denominator is rank).} Such objects have $\nu=+\infty$ so, in particular, $F_1$ can never be semi-destabilised: it is either stable or strictly unstable, and has no walls of instability. Since it is semistable on $\ell$, and this intersects $b=b_1$ at the point
$$
w_1\ =\ \frac{n}{2H^3} + \frac{\ch_2(F_1).H}{H^3}
$$
which defines a weak stability condition in $U$ by
$$
w_1 - \frac{{b}_1^2}{2}\ =\  \frac{n}{2H^3} - \frac{1}{2(H^3)^2} + \frac{\ch_2(F_1).H}{H^3}\ \ge\ \frac{nH^3-1-4H^3}{2(H^3)^2}\ >\ 0,
$$
we conclude the following.

	\begin{Lem}\label{large 1}
	    The destabilising sheaf $F_1$ is $\nu\_{b_1, w}$-stable for any $w > \frac{b_1^2}{2}. \hfill\square$
	\end{Lem}
  
  Similarly if we work with $b=b_2:= -n + \frac{1}{H^3}$ then the denominator of $\nu\_{b_2,w}$ is
  $$
  D_2(E)\ :=\ \ch_1(E(n)).H^2-\ch_0(E).
  $$
This has the same properties as $D_1(E)$, except the third is replaced now by $D_2(F_2)=1$ being minimal. Again $\ell$ intersects $b=b_2$ in a point
$$  
  w_2\ =\ \frac{n^2}{2} - \frac{n}{2H^3} + \frac{\ch_2(F_1).H}{H^3} 
$$
inside the space $U$ of weak stability conditions, by
$$
w_2 - \frac{{b}_2^2}{2}\ =\  \frac{n}{2H^3} - \frac{1}{2(H^3)^2} + \frac{\ch_2(F_1).H}{H^3}\ \ge\ \frac{nH^3-1-4H^3}{2(H^3)^2}\ >\ 0.
$$
   So the same argument as for Lemma \ref{large 1} gives the following.
	\begin{Lem}\label{large 2}
	   The destabilising quotient $F_2$ is $\nu\_{{b}_2, w}$-stable for any $w > \frac{{b}_2^2}{2}. \hfill\square$ 
	\end{Lem}  

   \begin{Prop}\label{first}
$\displaystyle{   	\ch_3(F_1)\ \leq\ \frac{2}{3} \ch_2(F_1).H\left(\!\ch_2(F_1).H - \frac{1}{2H^3}\right).}$
    \end{Prop}

\begin{proof}
Recall the line $\{b=b_1\}\cap U$ used in Lemma \ref{large 1}. Its base on $w=\frac{b^2}2$ is the point $\big(\!-\frac1{H^3},\frac1{2(H^3)^2}\big)$. Let $\ell_2$ denote the line connecting this point to $\Pi(F_1)=\big(0,\frac{\ch_2(F_1).H}{H^3}\big)$,
\begin{equation}\label{ell2}
   	w\ =\ \left(\ch_2(F_1).H - \frac{1}{2H^3} \right) b + \frac{\ch_2(F_1).H}{H^3}\,.
\end{equation}
   \begin{figure}[h]
	\begin{centering}
		\definecolor{zzttqq}{rgb}{0.27,0.27,0.27}
		\definecolor{qqqqff}{rgb}{0.33,0.33,0.33}
		\definecolor{uququq}{rgb}{0.25,0.25,0.25}
		\definecolor{xdxdff}{rgb}{0.66,0.66,0.66}
		
		\begin{tikzpicture}[line cap=round,line join=round,>=triangle 45,x=1.0cm,y=1.0cm]
		
		\draw[->,color=black] (-4,0) -- (4,0);
		\draw  (4, 0) node [right ] {$b,\,\frac{\ch_1.H^2}{\ch_0H^3}$};


		\fill [fill=gray!40!white] (0,0) parabola (2.95,3.87) parabola [bend at end] (-2.95,3.87) parabola [bend at end] (0,0);
		
		\draw  (0,0) parabola (3,4); 
		\draw  (0,0) parabola (-3,4); 
		\draw  (2.1, 1.9) node [right] {$w= \frac{b^2}{2}$};
		
		\draw  (0,-1) [color=red]parabola (2,4); 
		\draw  (0,-1) [color=red] parabola (-2,4); 
		\draw  (.4,-.8) node [right] [color=red]{$w= b^2 + \frac{\ch_2(F_1).H}{H^3}$};
		
		

		\draw[->,color=black] (0,-2) -- (0,5);
		\draw  (0, 5) node [above ] {$w,\,\frac{\ch_2.H}{\ch_0H^3}$};
		
	   \draw[color=black] (0.3,-1.5) -- (-3,4);
	   \draw[color=black, dashed] (-.73,0) -- (-.73,4);
	 \draw  (-.65, 3) node [left] {$b\!=$};
	 \draw  (-.8, 3) node [right] {$\frac{-1}{H^3}$};
		
		\draw[color=black, dashed] (-1.32,1.23) -- (0,1.23);
		\draw[color=black, dashed] (-1.32,1.23) -- (-1.32,0);

		\draw  (0.25,-1.6) node [right] {$\ell_2$};
		\draw  (1, 3) node [above] {\Large $U$};
		\draw  (0.1, -1.3) node [left] {$\Pi(F_1)$};
		\draw  (-1.32,0) node [below] {$b^*$};
		\draw  (0,1.23) node [right] {$w^*$};
		
		\begin{scriptsize}
		\fill (0, -1) circle (2pt);
		\fill (-.73,.24) circle (2pt);
		\fill (-1.32,1.23) circle (2pt);
		\fill (-1.32,0) circle (2pt);
		\fill (0,1.23) circle (2pt);

		
		
		\end{scriptsize}
		
		\end{tikzpicture}
		
		\caption{The first wall for the sheaf $F_1$}
		
		\label{first-wall}
		
	\end{centering}
\end{figure}
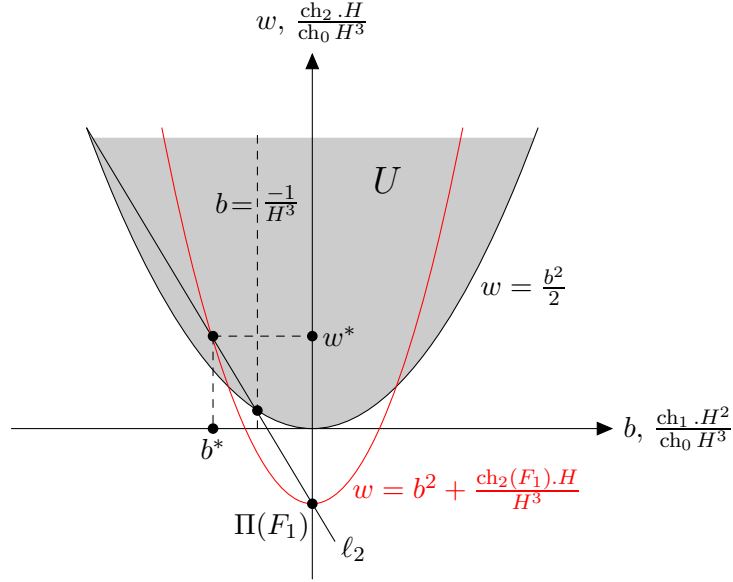	
	
	By the description of the walls of instability (Proposition \ref{locally finite set of walls}), the $w\downarrow\frac{b_1^2}2$ limit of Lemma \ref{large 1} therefore shows that $F_1$ is $\nu\_{b,w}$-semistable for any $(b,w)\in\ell_2\cap U$; see Figure \ref{first-wall}.
	
To apply the Bogomolov-Gieseker Conjecture \ref{conjecture} to $F_1$ on $\ell_2$ we need to find a point of $\ell_2\cap U$ satisfying $\ch_2^{bH}(F_1).H = \left(w- \frac{b^2}{2}\right)\ch_0(F_1)H^3$, i.e.
$$
   	\frac{\ch_2(F_1).H}{H^3} + \frac{b^2}{2}\ =\ w -\frac{b^2}{2}\,.  
$$
This intersects $\ell_2$ \eqref{ell2}
at the point $(b^*,w^*)$, where
\begin{equation*}
   	b^*\ =\ \ch_2(F_1).H - \frac{1}{2H^3} \quad\text{and}\quad w^*\ =\ \big(\!\ch_2(F_1).H\big)^2+ \frac1{4(H^3)^2}\,.
   	\end{equation*}
$\nu\_{b^*,w^*}$ is a weak stability condition since $w^*-\frac{(b^*)^2}2=\frac12\big(\!\ch_2(F_1).H+\frac1{2H^3}\big)^2>0$, so by \ref{BG2} we may apply Conjecture \ref{conjecture} to give
   	\begin{align*}
   		\ch_3(F_1) - b^*\ch_2(F_1).H - \frac{(b^*)^3 H^3}{6} &\ \leq\  \frac{1}{3}\left(w^*- \frac{(b^*)^2}{2}\right)(-{b^*}H^3)\\
   		&\ =\ \frac{1}{3}\left(\frac{\ch_2(F_1).H}{H^3} + \frac{(b^*)^2}{2}\right)(-{b^*}H^3).
   	\end{align*}
Simplifying gives
   	\begin{equation*}
\ch_3(F_1)\ \leq\ \frac{2}{3}\,b^*\ch_2(F_1).H. \qedhere
   	\end{equation*}
\end{proof}

   \begin{Prop}\label{second}
$\displaystyle{\ch_3(F_2(n))\ \leq\ \frac{2}{3} \ch_2(F_2(n)).H\left(\!\ch_2(F_2(n)).H + \frac{1}{2H^3}\!\right)\!.}$ 
   \end{Prop}
   
 \begin{proof}
By Lemma \ref{large 2},  $F_2 \in \cA(b_2)$ is $\nu\_{b_2, w}$-semistable for $w \gg 0$. Thus $F_2(n)\in \cA(b_2+n) =\cA(-b_1)$ is $\nu_{-b_1, w}$-semistable for $w \gg 0$.
Therefore, by \cite[Lemma 5.1.3(b)]{BMT} the shifted derived dual $F_2(n)^{\vee}[1]$ lies in an exact triangle 
  \begin{equation*}
     F\ \Into\, F_2(n)^{\vee}[1]\, \To\hspace{-5.5mm}\To\, Q[-1],
  \end{equation*}
with $Q$ a zero-dimensional sheaf and $F$ a $\nu\_{b_1,w}$-semistable object of $\cA(b_1)$ for $w \gg 0$. Since $\rk F=1$ it is a torsion-free sheaf by \cite[Lemma 2.7]{BMS}. We also have $\ch_1(F).H^2 = \ch_1(F_2(n)).H^2 = 0$. Thus $F$ has all the  properties of $F_1$ used in Lemma \ref{large 1} and Proposition \ref{first}, so the latter gives 
    \begin{equation*}
   	\ch_3(F)\ \leq\ \frac{2}{3} \ch_2(F).H\left(\ch_2(F).H - \frac{1}{2H^3}\right)\!. 
   	\end{equation*}
  Since $\ch_2(F_2(n)).H = -\ch_2(F).H$ and $\ch_3(F_2(n)) =\ch_3(F_2(n)^\vee[1]) = \ch_3(F) - \ch_3(Q) \leq \ch_3(F)$ the claim follows.
 \end{proof}


\subsection*{Proof of Proposition \ref{prop.case ii}}
Set $c:=\ch_2(F_1).H\in\{0,-1,-2\}$, so by Proposition \ref{first}, 
\begin{equation}\label{part.1}
    \ch_3(F_1)\ \leq\ \frac{2c}{3}\left(c- \frac{1}{2H^3} \right). 
\end{equation}
Using $\ch_0(F_1) =1,\ \ch_1(F_1).H^2 = 0$ and the exact triangle $F_1\to\iota_* L \to F_2$ we compute
$$
\ch_1(F_2(n)).H^2\,=\,0, \quad \ch_2(F_2(n)).H\, =\, -c \quad\mathrm{and}\quad \ch_3(F_2(n))\, =\, -nc -\ch_3(F_1) + \frac{L^2}{2}\,.
$$ 
The inequality of Proposition \ref{second} therefore becomes
\begin{equation*}
-nc -\ch_3(F_1) + \frac{L^2}{2}\ \leq\ -\frac{2c}{3}\left(-c + \frac{1}{2H^3}\right).
\end{equation*}
Combined with \eqref{part.1} and our assumption $L^2>-2n+4$ this gives
$$
-n(c+1)+2\ <\ -nc+ \frac{L^2}{2}\ \le\ \ch_3(F_1)-\frac{2c}3\left(-c + \frac{1}{2H^3}\right)\ \le\ \frac{4c}3\left(c- \frac{1}{2H^3} \right).
$$
If $c=-1$ this gives the contradiction $2<\frac43+\frac2{3H^3}$. If $c=-2$ we get $n+2<\frac{16}3+\frac4{3H^3}<7$ but $n\ge10$. So $c=0. \hfill\square$

	\section{Curve counting}\label{Sdual}
The results of this paper are a special case of the results in \cite{FT}, which in turn builds on \cite{GST}. Consider 2-dimensional torsion sheaves of the form $\iota_*(L\otimes I_C)$, where $D\in|\cO(n)|$ and $I_C\subset\cO_D$ is the ideal sheaf of a subscheme of dimension $\le1$. We take $L.H=0$ and $n$ sufficiently large as in this paper; the main difference in \cite{FT} is that we allow nonempty $C$.

We show the moduli space of slope semistable sheaves in the class of $\iota_*(L\otimes I_C)$ is isomorphic to the product of $\mathrm{Pic}_{\mathrm{tors}}(X)$\footnote{Note this is \emph{not} the set of torsion line bundles, though it contains it of course.} -- the line bundles on $X$ with torsion $c_1$ -- and the moduli space of Joyce-Song pairs
\begin{equation}\label{JSpair}
\cO(-n)\xrightarrow{\ s\ }\cI_C.
\end{equation}
Here $\cI_C\subset\cO_X$ is an ideal sheaf on $X$ and $s\in H^0(\cO(n))$ is a nonzero section with zero divisor $D\supset C$.
The correspondence takes the cokernel of \eqref{JSpair} and tensors it with a line bundle $L$ with torsion $c_1(L)$ to get a sheaf of the form $\iota_*(L\otimes I_C)$.

For $n\gg0$ the moduli space of pairs \eqref{JSpair} is a projective bundle over the moduli space of ideal sheaves $\cI_C$. The fibre $\PP\big(H^0(\cI_C(n))\big)$ has Euler characteristic $\chi(\cI_C(n))$. Allowing $D$ to be in the linear system $|L’(n)| $ for any $L’$ with $c_1(L’)=0\in H^2(X,\mathbb Q) $ adds another $H^2(X,\mathbb Z)_{\mathrm{tors}}$ factor. If $X$ is a Calabi-Yau 3-fold with $H^1(\cO_X)=0$ this gives the relation
$$
\#(\text{2-dimensional sheaves})\,=\,(-1)^{\chi(\cI_C(n))-1}\big(\#H^2(X,\Z)_{\mathrm{tors}}\big)^2\chi(\cI_C(n))\cdot\#(\text{ideal sheaves}).
$$
The first term is a DT invariant counting Gieseker stable sheaves\footnote{We show that for $n\gg0$, slope semistability is equivalent to slope stability and to Gieseker stability.} with the same Chern character as $\iota_*(L\otimes I_C)$. The next terms are topological constants. The final term is the DT invariant counting ideal sheaves of the topological type of $\cI_C$.

The set of all of these DT invariants counting ideal sheaves is equivalent, by the MNOP conjecture \cite{MNOP} (proved for most Calabi-Yau 3-folds in \cite{PP}), to the set of Gromov-Witten invariants of $X$. The upshot is that the Gromov-Witten invariants of $X$ are governed by counts of 2-dimensional sheaves. In turn the generating series of the latter are conjectured by physicists to be mock modular forms due to \emph{S-duality}.

\appendix
\section{The case of $\PP^3$ \\ \emph{By Claire Voisin}} \label{app}

When $X=\PP^3$ we can prove a very similar result to \emph{\ref{B}} by more classical methods.

\begin{Thm} Let $D$ be a smooth surface of degree $n\ge4$ in $\PP^3$. Any nontrivial line bundle $L$ on $D$ with $c_1(L).H=0$ satisfies $L^2\le-2n+5$.
\end{Thm}

\begin{proof} The $K3$ case $n=4$ is trivial: Riemann-Roch gives $h^0(L)+h^0(L^{-1})=h^1(L)+2+\frac{L^2}2$ so if $L$ is nontrivial with $L.H=0$ this gives $0=h^1(L)+2+\frac{L^2}2$ and so $L^2\le-4$.

So we can take $n\ge5$. By Riemann-Roch,
\begin{equation} \label{RR}
h^0(L)+h^0(K_D\otimes L^{-1})\ \ge\ \chi(L)\ =\ \chi(\cO_D)+\frac12 L^2-\frac12 K_D.L.
\end{equation}
We assume for a contradiction that $L$ is nontrivial and $L^2\ge-2n+6$. Using $K_D=\cO_D(n-4),\ L.H=0$ and $h^1(\cO_D)=0$, \eqref{RR} gives
\begin{equation} \label{1}
h^0(L^{-1}(n-4))\ \ge\ h^0(\cO_D(n-4))-(n-4).
\end{equation}
Let $C:=H\cap D$ be a smooth plane section and $L_C:=L|_C$. Then the exact sequences $0\to L^{-1}(i-1)\to L^{-1}(i)\to L^{-1}_C(i)\to0$  give
\begin{equation} \label{CC}
h^0\big(L^{-1}(i)\big)-h^0\big(L^{-1}(i-1)\big)\ \le\ h^0\big(L^{-1}_C(i)\big).
\end{equation}
Since $h^0(L^{-1})=0$, summing over $1\le i\le n-4$ gives
\begin{equation} \label{2}
h^0\big(L^{-1}(n-4)\big)\ \le\ \sum_{i=1}^{n-4}h^0\big(L^{-1}_C(i)\big).
\end{equation}
Replacing $L^{-1}$ by $\cO_D$ gives equality in \eqref{CC} for $1\le i\le n-4$ by Kodaira vanishing, and so
\begin{equation} \label{3.5}
h^0\big(\cO_D(n-4)\big)-1\ =\ \sum_{i=1}^{n-4}h^0\big(\cO_C(i)\big).
\end{equation}
Comparing \eqref{1}, \eqref{2} and \eqref{3.5} shows that $h^0\big(L^{-1}_C(i)\big)\ge h^0\big(\cO_C(i)\big)$ for some 
$1\le i\le n-4$. Since $\deg L_C=0$ this implies $L_C=\cO_C$ by \cite[Theorem 2.1, 2(b)]{Ha}.

By standard methods, this now implies the contradiction $L=\cO_D$. For instance, consider the blow up $\pi\colon\widehat D\to D$ of $D$ in the baselocus of a pencil of $C$s, giving a fibration $p\colon\widehat{D}\to\PP^1$. Then $\pi^*L$ is trivial on the fibres, so is the pullback from $\PP^1$ of the line bundle $p_*(\pi^*L)\cong\cO_{\PP^1}(d)$. Restricting $\pi^*L$ to (the proper transform of) another plane section (a multisection of $p$) and using $L.H=0$ shows that $d=0$.
\end{proof}

	\bibliographystyle{halphanum}
\bibliography{references}

\begin{thebibliography}{BMS16}
	

\bibitem[BMS16]{BMS} A. Bayer, E. Macr\`i and P. Stellari, \emph{The space of stability conditions on abelian threefolds, and on some Calabi-Yau threefolds}, Invent. Math. \textbf{206} (2016), 869--933. \arXiv{1410.1585}.

\bibitem[BMT14]{BMT} A. Bayer, E. Macr\`i and Y. Toda, \emph{Bridgeland stability conditions on threefolds I: Bogomolov-Gieseker type inequalities}, Jour. Alg. Geom. \textbf{23} (2014), 117--163.
\arXiv{1103.5010}.

\bibitem[Br08]{Br} T. Bridgeland, \emph{Stability conditions on $K3$ surfaces}, Duke Math. Jour. \textbf{141} (2008), 241--291. \arXiv{0307164}.

\bibitem[Br07]{Br.stbaility} T. Bridgeland, \emph{Stability conditions on triangulated categories}, Ann. of Math. \textbf{166} (2007), 317--345. \arXiv{0212237}.

\bibitem[FT20]{FT} S. Feyzbakhsh and R. P. Thomas, \emph{Curve counting and S-duality}, preprint.

\bibitem[GST14]{GST} A. Gholampour, A. Sheshmani and R. P. Thomas, \emph{Counting curves on surfaces in Calabi-Yau 3-folds}, Math. Ann. \textbf{360}, 67--78, 2014. \arXiv{1309.0051}.

\bibitem[Ha86]{Ha} R. Hartshorne, \emph{Generalized divisors on Gorenstein curves and a theorem of Noether}, J. Math. Kyoto Univ \textbf{26} (1986), 375--386.

\bibitem[Ko18a]{KosekiAb} N. Koseki, \emph{Stability conditions on product threefolds of projective spaces and Abelian varieties}, Bull. LMS \textbf{50} (2018), 229--244. \arXiv{1703.07042}.

\bibitem[Ko18b]{Kosekinef} N. Koseki, \emph{Stability conditions on threefolds with nef tangent bundles}, \arXiv{1811.03267}.

\bibitem[Li19a]{Li.Fano} C. Li, Stability conditions on Fano threefolds of Picard number one, J. Eur. Math. Soc. \textbf{21} (2019), 709--726. \arXiv{1510.04089}.

\bibitem[Li19b]{Li} C. Li, \emph{On stability conditions for the quintic threefold}, Invent. Math. \textbf{218} (2019), 301--340. \arXiv{1810.03434}.

\bibitem[MP16]{MP} A. Maciocia and D. Piyaratne, \emph{Fourier-Mukai Transforms and Bridgeland Stability Conditions on Abelian Threefolds II}, Int. Jour. of Math. \textbf{27} (2016), 1650007. \arXiv{1310.0299}.

\bibitem[Ma14]{Ma} E. Macr\`i, \emph{A generalized Bogomolov-Gieseker inequality for the three-dimensional projective space}, Algebra \& Number Theory \textbf{8} (2014), 173--190.
\arXiv{1207.4980}.

\bibitem[MNOP]{MNOP}
D.~Maulik, N.~Nekrasov, A.~Okounkov, and R.~Pandharipande.
\newblock {\em Gromov-{W}itten theory and {D}onaldson-{T}homas theory, {I}},
  Compos. Math., {\bf 142} (2006), 1263--1285.
\newblock \mathAG{0312059}.

\bibitem[PP17]{PP} R. Pandharipande and A. Pixton, {\em Gromov-Witten/Pairs
correspondence for the quintic 3-fold}, Jour. AMS \textbf{30} (2017), 389--449.

\bibitem[Sc17]{Sc} B. Schmidt, \emph{Counterexample to the generalized Bogomolov-Gieseker inequality for threefolds}, Int. Math. Res. Notices \textbf{2017} (2017), 2562--2566.
\arXiv{1602.05055}.

\bibitem[Sc14]{ScQ} B. Schmidt, \emph{A generalized Bogomolov-Gieseker inequality for the smooth quadric threefold}, Bull. LMS \textbf{46} (2014), 915--923. \arXiv{1309.4265}.

\bibitem[To12]{TodaBG} Y. Toda. \emph{Bogomolov-Gieseker-type inequality and counting invariants}, Jour. of Topol. \textbf{6} (2012), 217--250. \arXiv{1112.3411}.


\bibitem[Vo07]{Voi} C. Voisin, \emph{Some aspects of the Hodge conjecture}, Japanese Jour. of Math. \textbf{2} (2007), 261--296.	
	
\end{thebibliography}

\bigskip \noindent {\tt{s.feyzbakhsh@imperial.ac.uk\\ richard.thomas@imperial.ac.uk}}

\noindent Department of Mathematics\\
\noindent Imperial College\\
\noindent London SW7 2AZ \\
\noindent United Kingdom 
\bigskip

\noindent {\tt{claire.voisin@imj-prg.fr}}\\
\noindent Coll\` ege de France\\
\noindent 3 rue d'Ulm\\
\noindent 75005 Paris\\
\noindent France

\end{document}